\documentclass[psamsfonts,12pt]{amsart}
\usepackage[margin=1in]{geometry}  
\usepackage{graphicx}              
\usepackage{amsmath}               
\usepackage{amsfonts}              
\usepackage{amsthm}                
\usepackage{verbatim}              
\usepackage{indentfirst}           
\usepackage{underscore}
  \usepackage{color}
\usepackage{enumitem}
\usepackage{relsize}
\usepackage{mathtools} 
\usepackage{hyperref}
\hypersetup{
    colorlinks=true,
    linkcolor=blue,
    filecolor=magenta,      
    urlcolor=cyan,
}
\usepackage{amssymb}
\usepackage[all]{xy}
\makeatletter
\renewcommand*\env@matrix[1][*\c@MaxMatrixCols c]{%
  \hskip -\arraycolsep
  \let\@ifnextchar\new@ifnextchar
  \array{#1}}
\makeatother

\newtheorem{thm}{Theorem}[section] 
\newtheorem{lem}[thm]{Lemma}
\newtheorem{prop}[thm]{Proposition}
\newtheorem{cor}[thm]{Corollary}
\newtheorem{conj}[thm]{Conjecture}

\theoremstyle{remark}
       \newtheorem{rmk}{Remark}
\theoremstyle{remark}
       
\newtheorem{definition}{Definition}

\newcommand*{\Scale}[2][4]{\scalebox{#1}{$#2$}}%

\newcommand{\Mod}[1]{\ (\textup{mod}\ #1)}     
\newcommand{\QQ}{\mathbb{Q}}

 \newcommand{\Orb}{\operatorname{Orb}}
  \newcommand{\Gal}{\operatorname{Gal}}

 \newcommand{\ord}{\operatorname{ord}}  
  \newcommand{\GL}{\operatorname{GL}}  
    \newcommand{\Res}{\operatorname{Res}}

 \newcommand{\Aut}{\operatorname{Aut}}

\usepackage{amssymb,fge}
\newcommand{\mysetminus}{\mathbin{\fgebackslash}}

\title{Classifying Galois groups of small iterates via rational points} 
\author[Wade Hindes]{Wade Hindes}
\address{Department of Mathematics, The Graduate Center, City University of New York (CUNY); 365 Fifth Avenue, New York, NY 10016, USA}
\email{whindes@gc.cuny.edu}
\date{\today}
\begin{document}
\maketitle
\begin{abstract} \normalsize We establish several surjectivity theorems regarding the Galois groups of small iterates of $\phi_c(x)=x^2+c$ for $c\in\mathbb{Q}$. To do this, we use explicit techniques from the theory of rational points on curves, including the method of Chabauty-Coleman and the Mordell-Weil sieve. For example, we succeed in finding all rational points on a hyperelliptic curve of genus $7$, with rank $5$ Jacobian, whose points parametrize quadratic polynomials with a ``newly small" Galois group at the fifth stage of iteration.  
\end{abstract}
\renewcommand{\thefootnote}{}
\section{Introduction} 
Let $K$ be a global field of characteristic zero, let $\phi(x)\in K(x)$ be a rational function of degree $d\geq2$, and let $\phi^n$ denote the $n$th iterate of $\phi$. For basepoints $b\in K$, it is a major goal of arithmetic dynamics to understand the Galois groups 
\[G_n(\phi,b):=\Gal_K(\phi^n(x)-b).\]
For \emph{generic} $b\in K$, i.e. points where the equations $\phi^n(x)=b$ have $d^n$ distinct solutions in $\overline{K}$ for all $n\geq1$, one can view the Galois group $G_n(\phi,b)$ as a subgroup of the automorphism group of the $n$th-\emph{level preimage tree}, 
\[T_n(\phi,b):=\bigsqcup_{m=0}^n \phi^{-n}(b),\]
where the edge relation is given by evaluation of $\phi$. In particular, for generic baspoints, $T_n(\phi,b)$ is isomorphic (as a graph) to $T_{d,n}$, the $d$-ary rooted tree with $n$ levels, and is therefore dependent only on the degree of the map. Moreover, since the relevant splitting fields are nested, i.e. $K(T_m(\phi,b))\subseteq K(T_n(\phi,b))$ for $m\leq n$, we may form the inverse limit 
\[ G_\infty(\phi,b):=\lim_{\longleftarrow}G_n(\phi,b)\] 
with respect to the restriction maps. Finally, since the natural action of $G_\infty(\phi,b)$ on each $T_n(\phi,b)$ is compatible with restriction and inclusion, we obtain an injection 
\[ G_\infty(\phi,b)\leq \Aut(T_d),\] 
where $T_d:=\bigsqcup\, T_{d,n}$ is the full $d$-ary rooted preimage tree of $b$ with respect to $\phi$. This inclusion is known as the \emph{arboreal representation} \cite{Jonessurvey} associated to the pair $(\phi,b)$. 

The arboreal representations we have defined are dynamical analogs of the classical $\ell$-adic representations attached to elliptic curves (or more generally, abelian varieties) \cite{Serre}, where one instead appends to the ground field iterated preimages of the identity with respect to multiplication by $\ell$ maps. 

However, in practice there is a key difference between these two types of representations: to prove the surjectivity of the $\ell$-adic  representations attached to elliptic curves, it suffices to prove surjectivity onto some finite quotient. Namely, if $G\leq\GL_2(\mathbb{Z}_\ell)$ is a closed subgroup that surjects onto $\GL_2(\mathbb{Z}/\ell^n\mathbb{Z})$ for some small $n$, then $G$ must be equal to $\GL_2(\mathbb{Z}_\ell)$; see \cite{ladicsurj}. However, the analogous statement fails for closed subgroups $G\leq\Aut(T_d)$. Nonetheless, the author has salvaged a version of this ``surjective-rigidity" for subgroups $G\leq\Aut(T_2)$ coming from arboreal representations of quadratic polynomials defined over function fields of characteristic zero. \textbf{In what follows, we suppress the basepoint $b$ when $b=0$}. \vspace{.1cm}   
\begin{thm}[\cite{Me:GaloisUnif}] Let $k$ be a field of characteristic zero, and let $K=k(t)$ be a rational function field. For quadratic polynomials $\phi(x)\in K[x]$, write 
\[\phi(x)=\big(x-\gamma(t)\big)^2+c(t)\] for some $\gamma, c\in k[t]$ by completing the square. If $\deg(\gamma-c)=\max\{\deg(\gamma), \deg(c)\}$, then \vspace{.075cm}
\[G_{17}(\phi)=\Aut(T_{2,17})\;\;\;\text{implies}\;\;\; G_\infty(\phi)=\Aut(T_2). \vspace{.075cm}\] 
That is, if $\gamma$ and $c$ do not have the same leading term, the arboreal representation associated to $\phi$ is surjective if and only if it is surjective at the $17$th stage of iteration.\vspace{.1cm}   
\end{thm} 
Moreover, assuming the Vojta conjecture for curves, an analogous statement holds for specializations within a one-dimensional family.
Specifically, let $\phi(x)=(x-\gamma(t))^2+c(t)$ for some polynomials $\gamma(t),c(t)\in\mathbb{Q}[t]$, and let \vspace{.075cm} 
\[\phi_a(x)=(x-\gamma(a))^2+c(a)\in\mathbb{Q}[x]\;\;\;\;\;\;\;\text{for}\; a\in\mathbb{Z}\vspace{.075cm}\]
be the specialization of $\phi$ at $a$. Then we have the following theorem: \vspace{.1cm} 
\begin{thm}[\cite{Me:GaloisRig}]{\label{rigid}} Suppose that $\phi(x)\in\mathbb{Q}(t)[x]$ is non-isotrivial. If $\phi(\gamma)\cdot \phi^2(\gamma)\neq0$ and the Vojta conjecture \cite[Conjecture 25.1]{Vojta} holds, then there is an integer $n_\phi$ such that \vspace{.075cm} 
\[G_{n_\phi}(\phi_a)=\Aut(T_{2,n_\phi})\;\;\; \text{implies}\;\;\;G_\infty(\phi_a)=\Aut(T_2)\vspace{.025cm},\]
for all $a\in\mathbb{Z}$.\vspace{.025cm} 
\end{thm} 
Roughly speaking, the Vojta conjecture predicts that if the full arboreal representation is not surjective for some $\phi_a$ and $a\in\mathbb{Z}$, then one should be able to detect this early. In this paper, we gather evidence for this philosophy for the simplest non-trivial family of quadratic polynomials, \[\phi_c(x)=x^2+c\;\; \text{for}\;\; c\in\mathbb{Z},\vspace{.075cm}\] 
at several basepoints. In particular, when taking preimages of $0$, we obtain:   \vspace{.1cm}
\begin{thm}{\label{thm:small5th}} Let $\phi_c(x)=x^2+c$ for some $c\in\mathbb{Z}$. Then the following statements hold: 
\begin{enumerate}[topsep=8pt, partopsep=8pt, itemsep=10pt] 
\item[\textup{(1)}] If $G_{3}(\phi_c)=\Aut(T_{2,3})$, then $G_{5}(\phi_c)=\Aut(T_{2,5})$. 
\item[\textup{(2)}] If $c\neq3$ and $G_{2}(\phi_c)=\Aut(T_{2,2})$, then $G_{5}(\phi_c)=\Aut(T_{2,5})$.
\end{enumerate}
That is, outside of the single counterexample $c=3$, Galois-maximality at the $2$nd stage of iteration implies maximality up to the $5$th stage.
\end{thm} 
\begin{rmk} If $c=3$, then $G_{2}(\phi_c)=\Aut(T_{2,2})$ and $G_{3}(\phi_c)\leq\Aut(T_{2,3})$ has index $2$. 
\end{rmk}  
Furthermore, since $\Aut(T_{2,2})\cong D_4$ is a small concrete group, we can make Theorem \ref{thm:small5th} completely explicit: 
\begin{cor}{\label{cor:5th}} If $c\in\mathbb{Z}\mysetminus\{3\}$ is such that $-c$ and $-(c+1)$ are non-squares in $\mathbb{Z}$, then $G_{5}(\phi_c)=\Aut(T_{2,5})$.
\end{cor}
\begin{rmk} Note that $\#\Aut(T_{2,5})=2^{2^5-1}=2147483648$. Hence, it is difficult to compute $G_5(\phi_c)$ directly for large values of $c\in\mathbb{Z}$, underscoring the usefulness of Corollary \ref{cor:5th}.  
\end{rmk} 
Moreover, since a point search on the relevant curves associated to larger iterates yields no unknown rational points (coupled with Theorems \ref{rigid} and \ref{thm:small5th}) we conjecture that $n_\phi=3$ for this particular family:   
\begin{conj} Let $\phi_c(x)=x^2+c$,\; $c\in\mathbb{Z}$. If $G_3(\phi_c)=\Aut(T_{2,3})$ then $G_\infty(\phi_c)=\Aut(T_2)$. 
\end{conj} 
\begin{rmk} It follows from the main results of \cite{Stoll-Galois} that $G_1(\phi_c)=\Aut(T_{2,1})$ already implies $G_\infty(\phi_c)=\Aut(T_2)$ for $c\in\mathbb{Z}$ in certain congruence classes (with parity conditions).   
\end{rmk} 
Likewise, we obtain an analogous result (up to the $4$th iterate) for rational values $c\in\mathbb{Q}$ when we instead take preimages of the basepoint $b=1$: \vspace{.15cm}         
\begin{thm}{\label{thm:4th}} Let $\phi_c(x)=x^2+c$ for some $c\in\mathbb{Q}$. Then the following statements hold: 
\begin{enumerate}[topsep=8pt, partopsep=8pt, itemsep=10pt] 
\item[\textup{(1)}]  If $c\in\mathbb{Q}\mysetminus\big\{2, \frac{-4\,}{\;3}\big\}$ and $G_3(\phi_c,1)=\Aut(T_{2,3})$, then $G_4(\phi_c,1)=\Aut(T_{2,4})$. 
\item[\textup{(2)}]  If $c\in\mathbb{Z}\mysetminus\{-1,2\}$ and $G_2(\phi_c,1)=\Aut(T_{2,2})$, then $G_4(\phi_c,1)=\Aut(T_{2,4})$. \vspace{.15cm}     
\end{enumerate} 
\end{thm}
\begin{rmk} If $c= -1$,  then $G_{2}(\phi_c,1)=\Aut(T_{2,2})$ and $G_{3}(\phi_c,1)\leq\Aut(T_{2,3})$ has index $2$. Similarly, when $c\in\{2,-4/3\}$, then $G_{3}(\phi_c,1)=\Aut(T_{2,3})$ and $G_{4}(\phi_c,1)\leq\Aut(T_{2,4})$ have index $4$.  
\end{rmk} 
To prove Theorem \ref{thm:small5th} and Theorem \ref{thm:4th}, one must determine a complete set of rational (or integral) points for several hyperelliptic curves of moderate genus. This requires quite a bit of work, since each of the relevant curves has a non-trivial set of such points, due to the existence of post-critically finite polynomials (which are known to have non-surjective arboreal representations; see \cite[Theorem 3.1]{Jonessurvey}). As such, we rely heavily upon the computer algebra systems \texttt{Magma} and \texttt{Sage} throughout this paper; scripts providing the computational details for our work can be found in \texttt{Gal_Sml_Iter} at:
\\[8 pt]  
\indent\indent\indent\indent\indent\indent \boxed{\text{\url{https://sites.google.com/a/alumni.brown.edu/whindes/research}}  \vspace{.15cm}}.  
\\[8 pt]
Finally, in section \ref{sec:irre} we discuss the possibility that irreducibility is a ``rigid" property of quadratic arboreal representations, i.e whether there is a \emph{uniform} iterate $N\geq1$ such that:
\[\text{If $\phi^N$ is irreducible over $\QQ$, then all iterates of $\phi$ are irreducible over $\QQ$;}\] 
here $\phi(x)\in\mathbb{Z}[x]$ is any monic, quadratic polynomial. 
\\[7 pt]
\indent \textbf{Acknowledgements:} It is a pleasure to thank Michael Stoll for his help with several computations within this paper, especially the Mordell-Weil sieve calculation in Lemma \ref{lem:d=5} and the Coleman integrals in Lemma \ref{lem:5th2}. We also thank the anonymous referee for their helpful comments.        
\section{Relating Galois Groups and Rational Points}{\label{sec:Galois}} 
We begin by showing how one passes from information about Galois groups of iterates of quadratic polynomials to rational points on curves. In what follows, $K$ is a global field of characteristic not $2$. The main idea is the following: write $\phi(x)=(x-\gamma)^2+c$ for some $\gamma,c\in K$. Then $G_n(\phi)=\Aut(T_{2,n})$, unless there is a multiplicative dependence relation, modulo squares, in the \emph{adjusted, post-critical set}:  
\begin{equation}{\label{post-crit}}
\{-\phi(\gamma),\phi^2(\gamma), \phi^3(\gamma),\dots, \phi^n(\gamma)\}.
\end{equation}  
It is for this reason that the curves of interest are hyperelliptic. This fact, known thus far only for the polynomials $\phi(x)=x^2+c$,  is a consequence of a few separate results (all generalizations of those in \cite[\S1.]{Stoll-Galois}), the first of which we state from \cite[Proposition 4.2]{Rafe:LMS}.  
\begin{prop}{\label{lem:irre}} Let $\phi(x)=(x-\gamma)^2+c$ be a quadratic polynomial defined over $K$. If the adjusted, post-critical set in \textup{(\ref{post-crit})} contains no squares in $K$, then $\phi^n$ is irreducible over $K$.    
\end{prop} 
\begin{rmk} Note that if $\phi^m(x)=f(x)\cdot g(x)$, then $\phi^{m+1}(x)=f(\phi(x))\cdot g(\phi(x))$; in particular, it follows that if $\phi^n(x)$ is irreducible, then all lower order iterates are also irreducible. 
\end{rmk}
Let $K_n(\phi)=K(\phi^{-n}(0))$ be the splitting field over $K$ of $\phi^n(x)$ for $n\geq1$. Then it is straightforward to verify that $K_n(\phi)/K_{n-1}(\phi)$ is the compositum of at most $\deg(\phi^{n-1})=2^{n-1}$ quadratic extensions; see \cite[Fact 1]{Stoll-Galois}. Thus \vspace{.025cm}
\[H_n(\phi):=\Gal(K_n(\phi)/K_{n-1}(\phi))\cong(\mathbb{Z}/2\mathbb{Z})^m\vspace{.07cm}\] 
for some $0\leq m\leq 2^{n-1}$, and we say that $H_n(\phi)$ is \emph{maximal} if $m=2^{n-1}$ is as large as possible. The following maximality criterion (cf. \cite[Lemma 1.6]{Stoll-Galois}) can be found in \cite[Lemma 3.2]{Rafe:LMS} in full generality.  
\begin{lem}{\label{lem:Stoll}} Let $\phi(x)=(x-\gamma)^2+c$ be a quadratic polynomial defined over $K$, and suppose that $\phi^n$ is irreducible over $K$ for all $n\geq1$. Then for $n\geq2$, $H_n(\phi)$ is maximal if and only if $\phi^n(\gamma)$ is not a square in $K_{n-1}(\phi)$.    
\end{lem}
It is well known that $G_n(\phi)=\Aut(T_{2,n})$ if and only if $G_{n-1}(\phi)=\Aut(T_{2,n-1})$ and $H_n(\phi)$ is maximal; see, for instance, \cite[Lemma 1.4]{Stoll-Galois}. Therefore, in light of Lemma \ref{lem:Stoll}, we need to test whether $\phi^n(\gamma)$ is a square in $K_{n-1}(\phi)$ for all $n$, to determine if the arboreal representation is surjective. To do this, we generalize \cite[Lemma 1.5]{Stoll-Galois}. However, in order to provide a clean statement, we fix some notation. 
\begin{definition} Set $c_1=-\phi(\gamma)$ and $c_n=\phi^n(\gamma)$ for $n\geq 2$. If $c_m\in K^*$ for all $m\geq1$, then we say that $c_1,c_2,\dots c_n$ are \emph{2-independent}, if their residue classes in the $\mathbb{F}_2$-vector space $K^*/(K^*)^2$ are linearly independent (equivalently, they generate a subspace of dimension $n$). Moreover, if some $c_i=0$, then we say that $c_1,c_2,\dots c_n$ are not $2$-independent. 
\end{definition}
With this definition in place, we can state the following criterion for determining whether or not an element $a\in K^*$ is a square in the splitting field $K_m$. Although it is only a minor generalization of \cite[Lemma 1.5]{Stoll-Galois}, we provide a proof here, since it does not seem to appear elsewhere in the literature.  
\begin{lem}{\label{2-ind}} If $G_m(\phi)=\Aut(T_m)$ and $c_1,c_2,\dots c_m$ are 2-independent, then for all $a\in K^*$ \textup{:} 
\[ a\not\in (K_m^*)^2 \Longleftrightarrow c_1,c_2,\dots, c_m, a\; \text{are $2$-independent}.\] 
\end{lem}  
\begin{proof} We note that the largest abelian quotient, or abelianization, of $\Aut(T_{2,m})$ is $(\mathbb{Z}/2\mathbb{Z})^m$. To see this, one can prove inductively (see \cite[p.215]{Wreath}) that $(G[H])^{\text{ab}}=G^{\text{ab}}\times H^{\text{ab}}$, where $G$ is a transitive permutation group, $G[H]$ is the wreath product of $G$ and $H$, and the superscript `ab' is the abelianization of a group, and then use the identification $[\mathbb{Z}/2\mathbb{Z}]^m\cong\Aut(T_{2,m})$ of the automorphism group of $T_{2,m}$ as an $m$-fold wreath product of cyclic groups. In any case, we deduce that the largest $2$-Kummer extension of $K$ within $K_m$ has degree $2^m$, and the kernel of the natural map $K^*/(K^*)^2\rightarrow  K_m^*/(K_m^*)^2$ has $\mathbb{F}_2$-dimension $m$. 

Now let $R_s$ be the set of roots of $\phi^s(x)$ in $\overline{K}$ for $s\geq1$. Note that Lemma \ref{2-ind} is clear when $m=1$, since $K_1=K(\sqrt{c_1})$ is itself a quadratic extension. Therefore, we can assume that $m\geq2$ and write 
\begin{align*}
c_m&=\phi^m(\gamma)=\phi^{m-1}(\phi(\gamma))=\phi^{m-1}(c)\\ 
\\ 
&=\prod_{\mathclap{\alpha\in R_{m-1}}}(c-\alpha)=\prod_{\mathclap{\;\,\alpha\in R_{m-1}}}(-1)(\alpha-c)=(-1)^{\#R_{m-1}}\cdot \prod_{\mathclap{\;\,\alpha\in R_{m-1}}}(\alpha-c)=\prod_{\mathclap{\;\,\alpha\in R_{m-1}}}(\alpha-c), \vspace{.1cm}
\end{align*}  
since $\#R_{m-1}$ is a non-trivial power of $2$. However, if $\alpha\in R_{m-1}$, then $\alpha-c=(\beta-\gamma)^2$ for some $\beta\in R_m$. Therefore, $c_m$ is a square in $K_m$. Likewise, since $K_{1}\subseteq K_2\dots \subseteq K_m$ and $G_m(\phi)=\Aut(T_{2,m})$ implies $G_s(\phi)=\Aut(T_{2,s})$ for all $1\leq s\leq m$, we see that $c_1, c_2, \dots c_m$ are all squares in $K_m$ by induction. However, the $c_1,c_2, \dots c_m$ are $2$-independent by assumption, and hence $c_1,c_2, \dots c_m$ must generate the kernel of the map $K^*/(K^*)^2\rightarrow  K_m^*/(K_m^*)^2$. The asserted equivalence follows.         
\end{proof}
Combining these results, we obtain a generalization of \cite[Theorem 1]{Stoll-Galois}.  
\begin{thm}{\label{thm:surjective}} Let $\phi(x)=(x-\gamma)^2+c$ be a quadratic polynomial defined over $K$. Then $G_n(\phi)=\Aut(T_{2,n})$ if and only if $c_1, c_2, \dots, c_n$ are $2$-independent.  
\end{thm}
\begin{proof} Assume that $G_n(\phi)=\Aut(T_{2,n})$ and proceed by induction. The base case $n=1$ is clear, since $K_1=K(\sqrt{c_1})$ is a quadratic extension. Therefore, we may assume that $n\geq2$. In particular, $G_n(\phi)=\Aut(T_{2,n})$ implies that $G_{n-1}(\phi)=\Aut(T_{2,n-1})$ and $H_n(\phi)$ is maximal. Moreover, since $G_n(\phi)=\Aut(T_{2,n})$ acts transitively on the roots of $T_{2,n}$ at every level, $\phi^n$ must be an irreducible polynomial. Hence, Lemma \ref{lem:Stoll} implies that $c_n\not\in (K_{n-1})^2$. On the other hand, the induction hypothesis implies that $c_1, c_2, \dots c_{n-1}$ are $2$-independent, since $G_{n-1}(\phi)=\Aut(T_{2,n-1})$. In particular, Lemma \ref{2-ind} implies that the full set $c_1, c_2, \dots c_{n-1}, c_n$ are $2$-independent as desired. 

Conversely, suppose that the $c_1, c_2, \dots, c_n$ are $2$-independent, and proceed by induction. Again the base case $n=1$ is clear, since $[K_1:K]=2$ if and only if the discriminant of $\phi$, which one checks is $-4c=4c_1$, is a a square in $K$. Therefore, we may assume that $n\geq2$. Note first that $\phi^n$ must be irreducible over $K$ by Proposition \ref{lem:irre}. Moreover, since the subset $c_1, c_2, \dots, c_{n-1}$ consists of $2$-independent elements, we see that $G_{n-1}(\phi)=\Aut(T_{2,n-1})$ by the induction hypothesis. In particular, Lemma \ref{2-ind} implies that $c_n\not\in (K_{n-1})^2$. Therefore, $H_n(\phi)$ is maximal by Lemma \ref{lem:Stoll} (here we use irreducibility). However, $G_{n-1}(\phi)=\Aut(T_{2,n-1})$ and $H_n(\phi)$ maximal implies $G_{n}(\phi)=\Aut(T_{2,n})$, which completes the proof.                   
\end{proof}  
Finally, we state the following corollary, formulated in a way that allows us to prove Theorem \ref{thm:small5th} and Theorem \ref{thm:4th}. In what follows, we identify $\mathbb{F}_2$ with the set $\{0,1\}$.   
\begin{cor}{\label{cor:curve}} Let $\phi(x)=(x-\gamma)^2+c$ be a quadratic polynomial defined over $K$. If $n\geq2$ is such that $G_{n-1}(\phi)=\Aut(T_{2,n-1})$ and $G_{n}(\phi)\neq\Aut(T_{2,n})$, then
\begin{equation} -\phi(\gamma)^{\epsilon_1}\cdot\phi^2(\gamma)^{\epsilon_2}\cdot \phi^3(\gamma)^{\epsilon_3}\dots \,\phi^{n-1}(\gamma)^{\epsilon_{n-1}}\cdot\phi^n(\gamma)=y_n^2
\end{equation}
for some $y_n\in K$ and some vector of exponents $\epsilon=(\epsilon_i)\in\mathbb{F}_2^{n-1}$.     
\end{cor} 
We now use Corollary \ref{cor:curve} to relate the size of iterated Galois groups to hyperelliptic curves. Namely, if $\phi(x)=(x-\gamma(t))^2+c(t)$ is a one-parameter family of quadratic polynomials over $K(t)$ and $a\in K$ is a specialization such that $G_{n-1}(\phi_a)=\Aut(T_{2,n-1})$ and $G_{n}(\phi_a)\neq\Aut(T_{2,n})$ for some $n\geq2$, then there exists $\epsilon\in\mathbb{F}_2^{n-1}$ and $y_n\in K$ such that $(a,y_n)$ is a rational point on the affine curve \vspace{.1cm} 
\begin{equation}{\label{curvesfamilies}}
C_{\phi,n}^{(\epsilon)}:=\Big\{(t,y)\,\Big|\, y^2=-\phi(\gamma(t))^{\epsilon_1}\cdot\phi^2(\gamma(t))^{\epsilon_2}\cdot \phi^3(\gamma(t))^{\epsilon_3}\dots \,\phi^{n-1}(\gamma(t))^{\epsilon_{n-1}}\cdot\phi^n(\gamma(t))\Big\}. \vspace{.1cm}
\end{equation} 
\vspace{.1cm}
In particular, we use this perspective for $\phi(x)=x^2+t$ when $n=5$ to prove Theorem \ref{thm:small5th} and (after conjugating) the family $\rho(x)=(x+1)^2+t$ when $n=3,4$ to prove Theorem \ref{thm:4th}.              
\section{Dynamical Galois Groups: $\phi_c(x)=x^2+c$ \textup{and} $b=1$}{\label{sec:rationalpoints}
We begin with the family $\phi_c(x)=x^2+c$ and the basepoint $b=1$ (see Theorem \ref{thm:4th}) over $K=\mathbb{Q}$, since the case when $b=0$ requires a few extra tricks to cut down the number of curves one must consider; see Theorem \ref{thm:prime} below.  
\begin{proof}[(Proof of Theorem \ref{thm:4th})] 
In order to use the Galois theory results from section \ref{sec:Galois}, which are all stated in terms of preimages of zero, we must change variables. Namely, if $c'=c-1$ and $\rho_{c'}(x)=(x+1)^2+c'$, then there is an equality of splitting fields, 
\[\mathbb{Q}(\phi_c^{-n}(1))=\mathbb{Q}(\rho_{c'}^{-n}(0)),\] 
for all $n\geq1$; to see this, simply check that $\phi_c^n(\alpha)=1$ implies $\rho_{c'}^n(\beta)=0$ for $\beta=\alpha-1$. 

Therefore, after a linear change of variables, it suffices to classify the rational points on the hyperelliptic curves $C_{\rho,3}^{(\epsilon_3)}$ and $C_{\rho,4}^{(\epsilon_4)}$ for $\rho(x)=(x+1)^2+t$ as in (\ref{curvesfamilies}) to prove Theorem \ref{thm:4th}. Here $\epsilon_3\in\mathbb{F}_2^{2}$ and $\epsilon_4\in\mathbb{F}_2^{3}$ respectively. In particular, since $\gamma=-1$, we compute that \vspace{.1cm}  
\begin{align*}
-\rho(-1)&=-t,\\
\rho^2(-1)&=t^2 + 3t + 1,\\
\rho^3(-1)&=t^4 + 6t^3 + 13t^2 + 13t + 4,\\
\rho^4(-1)&= t^8 + 12t^7 + 62t^6 + 182t^5 + 335t^4 + 398t^3 + 299t^2 + 131t + 25.  \vspace{.1cm} 
\end{align*} 
Hence, we must find all of the rational points on the curves  \vspace{.1cm}
\[C_{\rho,3}^{(\epsilon)}: y^2=(-t)^{e_1}\cdot(t^2 + 3t + 1)^{e_2}\cdot (t^4 + 6t^3 + 13t^2 + 13t + 4)\]
and  
\begin{align*}
C_{\rho,4}^{(\epsilon)}: y^2=&(-t)^{e_1}\cdot(t^2 + 3t + 1)^{e_2}\cdot (t^4 + 6t^3 + 13t^2 + 13t + 4)^{e_3}\\
&(t^8 + 12t^7 + 62t^6 + 182t^5 + 335t^4 + 398t^3 + 299t^2 + 131t + 25); \vspace{.1cm}  
\end{align*}
here $e_1,e_2, e_3\in\mathbb{F}_2$. At first pass this seems rather difficult. For instance, when $\epsilon=(1,1,1)$, $C_{\rho,4}^{(\epsilon)}$ has genus $7$. However, if $g(t)$ and $f(t)$ are polynomials with integer coefficients (at least one of which has even degree), then the rational points on the curve 
\[Z: y^2=f(t)\cdot g(t)\] 
are covered by the rational points on the family of curves $\big(X_d,\pi_d:X_d\rightarrow Z\big)$: 
\[X^{(d)}:= \{(t,u,v): du^2=f(t),\;\;\; dv^2=g(t)\},\;\;\;\;  \pi_d(t,u,v):=(t,u\cdot v).\]
Furthermore, we can assume that $d$ is one of the finitely many integers supported on the primes dividing the resultant of $f$ and $g$; see \cite[\S2.3]{Stoll-RationalPts}. In the case at hand, \vspace{.1cm} 
\[\Res\big(\rho^j(-1),\rho^3(-1)\big)\in\{-1,4\} \;\;\;\text{and}\;\;\; \Res\big(\rho^i(-1),\rho^4(-1)\big)\in\{-1,-5,25\},\vspace{.1cm} \]
for $1\leq j\leq2$ and $1\leq i\leq3$. In particular, we reduce the proof of Theorem \ref{thm:4th} to computing the integral points on the elliptic curves \vspace{.1cm} 
\[E^{(d)}: \; dy^2=t^4 + 6t^3 + 13t^2 + 13t + 4,\;\;\;\;\; d\in\{\pm{1},\pm{2}\}\vspace{.1cm} \]
and the rational points on the genus-three curves \vspace{.1cm}  
\[F^{(d)}: \; dy^2=t^8 + 12t^7 + 62t^6 + 182t^5 + 335t^4 + 398t^3 + 299t^2 + 131t + 25,\;\;\;\;\; d\in\{\pm{1},\pm{5}\}.\vspace{.1cm} \]
The reduction here is that the second defining equation of $X^{(d)}$ already has finitely many integral (resp. rational) solutions, and so we can disregard the first equation entirely. 

Moreover, \texttt{Magma} \cite{Magma} has a built in function for determining the integral solutions to hyperelliptic, quartic equations. In particular, we compute that \vspace{.05cm}
\begin{align*} 
&E^{(1)}(\mathbb{Z})=\{(-3,\pm{1}),(0,\pm{2})\},\;\;\; E^{(-1)}(\mathbb{Z})=\{(-1,\pm{1})\},\\ 
& E^{(2)}(\mathbb{Z})=\{(-4,\pm{8})\},\;\;\;\;\;\;\;\;\;\;\;\;\;\;\; E^{(-2)}(\mathbb{Z})=\{(-2,\pm{2})\}. 
\end{align*} 
Moreover, among the associated quadratic polynomials, only \vspace{.05cm} 
\[\rho_{-2}(x)=(x+1)^2-2, \vspace{.05cm}\] 
corresponding to the points $(-2,\pm{2})\in E^{(-2)}(\mathbb{Z})$, has a newly small Galois group at the third stage of iteration, i.e. $G_2(\rho_{-2})=\Aut(T_{2,2})$ and $G_3(\rho_{-2})\neq\Aut(T_{2,3})$. After conjugating back to standard form $x^2+c$, we obtain the polynomial $x^2-1$ for the basepoint $b=1$, which appears in part (2) of Theorem \ref{thm:4th}. From here, the rest of the proof follows from: 
\begin{align*} 
F^{(1)}(\mathbb{Q})&=\{\infty_{+},\infty_-, (0,\pm{5}), (-3,\pm{1})\},\\ 
F^{(-1)}(\mathbb{Q})&=\{(-1,\pm{1}), (-2,\pm{1})\},\\ 
F^{(5)}(\mathbb{Q})&=\{(1,\pm{85})\}, \\ 
F^{(-5)}(\mathbb{Q})&=\{(-7/3,\pm{185}/81)\}. \vspace{.1cm}  
\end{align*}
In particular, we deduce that \vspace{.05cm} 
\[\rho_{1}(x)=(x+1)^2+1\;\;\;\text{and}\;\;\;\rho_{\frac{-7}{\;3}}(x)=(x+1)^2-7/3\vspace{.05cm}, \]
are the only quadratic polynomials of the form, $\rho_c(x)=(x+1)^2+c$ for $c\in\mathbb{Q}$, satisfying $G_3(\rho_c)=\Aut(T_{2,3})$ and $G_4(\rho_c)\neq\Aut(T_{2,4})$. Conjugating these polynomials back to standard form $x^2+c$, we find $x^2+1$ and $x^2-4/3$ for the basepoint $b=1$ as claimed.  
\end{proof} 
We determine $F^{(d)}(\mathbb{Q})$ for $d\in\{\pm{1},\pm{5}\}$ in Lemmas \ref{lem:d=1}, \ref{lem:d=-1}, \ref{lem:d=5} and \ref{lem:d=-5} below. However, since their proofs have nothing to do with dynamics and since the techniques we use might be of independent interest to those studying rational points on hyperelliptic curves, we replace the variable $t$ with the (more standard) variable $x$, and replace the (more complicated) names $F^{(1)}$, $F^{(-1)}$, $F^{(5)}$ and $F^{(-5)}$ with $C_1$, $C_2$, $C_3$ and $C_4$ respectively.           
\begin{lem}{\label{lem:d=1}} Let $C_1$ be the hyperelliptic curve given by \vspace{.05cm}  
\[C_1: y^2=x^8 + 12x^7 + 62x^6 + 182x^5 + 335x^4 + 398x^3 + 299x^2 + 131x + 25.\vspace{.05cm}\] 
Then $C_1(\mathbb{Q})=\{\infty_{+},\infty_-, (0,\pm{5}), (-3,\pm{1})\}$ is a complete list of rational points. 
\end{lem} 
\begin{proof} Let $J_1$ be the Jacobian of $C_1$. A two-descent with \texttt{Magma} shows that $J_1(\mathbb{Q})$ has rank at most two. Moreover, we check that 
\[\gcd\big(\#J_1(\mathbb{F}_{11}),\#J_1(\mathbb{F}_{29})\big)=1,\] 
and hence $J_1(\mathbb{Q})$ has trivial torsion \cite[Appendix]{Katz}. On the other hand, the points 
\[P_1=\big[\infty_--\infty_+\big]\;\;\;\text{and}\;\;\;Q_1=\big[(-3,1)-\infty_+\big] \vspace{.1cm}\]  
generate a subgroup isomorphic to $\mathbb{Z}/2\mathbb{Z}\times\mathbb{Z}/1396\mathbb{Z}$ inside the reduction $J_1( \mathbb{F}_{13})$. Therefore, $J_1(\mathbb{Q})$ is free of rank two, and $P_1$ and $Q_1$ must generate a finite index subgroup. 

We proceed now with the method of Chabauty and Coleman at the prime $p=3$; see \cite{Poonen-McCal} for a nice introduction to this method. 

Let $\omega\in\Omega_{J_1}^1(\mathbb{Q}_3)\cong\Omega_{C_1}^1(\mathbb{Q}_3)$ be a regular $1$-form whose integral annihilates the Mordell-Weil group $J_1(\mathbb{Q})$; such a $1$-form is guaranteed to exist since the genus of $C_1$ is strictly larger than the rank of its Jacobian. To calculate $\omega$ explicitly (for this curve and all others) it is convenient to find independent divisor classes (up to \emph{linear} equivalence) $D_1,D_2\in J_1(\mathbb{Q})$ such that 
\[D_i=[p_{i1}-q_{i1}]+[p_{i2}-q_{i2}]+[p_{i3}-q_{i3}]\]  
for some points $p_{ij},q_{ij}\in C_1(\mathbb{Q}_3)$ satisfying $\pi_3(p_{ij})=\pi_3(q_{ij})$, where $\pi_3:C_1(\mathbb{Q}_3)\rightarrow C_1(\mathbb{F}_3)$ is the reduction map. Given divisors classes $D_i$ in this form, we can compute 
\[\int_0^{D_i}\omega=\int_{q_{i1}}^{p_{i1}}\omega+\int_{q_{i2}}^{p_{i2}}\omega+\int_{q_{i3}}^{p_{i3}}\omega,\]
using a simple $3$-adic parametrization of each residue class (i.e. fiber of the reduction map), which is implemented in \texttt{Sage} \cite{Sage}. However, because our curve has an even degree model, finding the divisor classes $D_1$ and $D_2$ requires a bit of extra work (due to the group law implementation for $J_1$ in \texttt{Magma}). 

For $C_1$, the first divisor is easy to find: $D_1=[(0,-5)-(-3,-1)]$. To find the other, we compute two auxiliary divisors $R_1=30Q_1-D_1$ and $R_2=30P_1+10Q_1$ in the kernel of the reduction map mod $3$ on $J_1$. The Mumford representations \cite{Mumford} of the $R_i=[a_i(x),b_i(x)]$ are much too large to write down here (requiring several pages). However, \texttt{Magma} readily verifies that the polynomials $a_1(x)$ and $a_2(x)$ split completely in $\mathbb{Q}_3[x]$, so that \vspace{.075cm}
\[R_1=[p_1+p_2+p_3+\infty_-]-2[\infty_-+\infty_+]\;\;\;\; \text{and}\;\;\;\; R_2=[q_1+q_2+q_3+\infty_-]-2[\infty_-+\infty_+] \vspace{.05cm}\]
for some points $p_i,q_i\in C_1(\mathbb{Q}_3)$. In particular, their difference is supported away from infinity: 
\[D_2=R_1-R_2=[p_1-q_1]+[p_2-q_2]+[p_3-q_3].\]
Moreover, we can rearrange the terms (if necessary) so that $\pi_3(p_{i})=\pi_3(q_{i})$, since $D_2$ is in the kernel of the reduction map.  

Now that we have our two divisors $D_1$ and $D_2$ we can begin to compute $\omega$ up to any precision. To do this, we use the standard basis $\eta_k=x^k\,dx/2y$ for $0\leq k\leq 2$ of $\Omega_{C_1}^1(\mathbb{Q}_3)$, and we compute with the \texttt{coleman_integral()} function in \texttt{Sage} that \vspace{.075cm}
\[ \Big(\int_0^{D_1}\eta_k\Big)_{0\leq k\leq2}=\Big(\,\int_{(-3,-1)}^{(0,5)}\eta_k\,\Big)_{0\leq k\leq2}=\big(\,2\cdot3 + 3^2 + O(3^3),\; O(3^3), \; 2\cdot3^2 + O(3^3\,) \big)\] 
and that 
\[\Big(\int_0^{D_2}\eta_k\Big)_{0\leq k\leq2}=\Big(\, \sum_{i=1}^3\int_{q_i}^{p_i}\eta_k \,\Big)_{0\leq k\leq2}=\big(\,2\cdot3^2 + O(3^3),\; 3^2 + O(3^3),\; 2\cdot3^2 +O(3^3)\,\big).\vspace{.1cm}\]
Moreover, $\omega=(c_2x^2+c_1x+c_0)\, dx/2y$ for some $c_0,c_1, c_2\in\mathbb{Z}_3$ and $\int_0^{D_1}\omega=\int_0^{D_2}\omega=0$ by definition of $\omega$. In particular, after scaling appropriately, we see that $c_1\equiv0\Mod{3}$ and $c_2\equiv c_1\Mod{3}$. Hence, the reduction of the annihilating differential $\omega$ must be 
\[ \overline{\omega}=\frac{(x^2+x)dx}{2y}\in\Omega_{C_1}^1(\mathbb{F}_3),\]
up to an irrelevant scalar. On the other hand, 
\[C_1(\mathbb{F}_3)=\big\{\overline{\infty}_+,\overline{\infty}_-, (0,\pm{2})\big\},\]  
and we are in the fortuitous situation that the rational points $C_1(\mathbb{Q})$ surject onto $C_1(\mathbb{F}_3)$. In particular, the residue classes of $\overline{\infty}_+$ and $\overline{\infty}_-$ each contain exactly one rational point, since $\ord_{\overline{\omega}}(\overline{\infty}_+)=0=\ord_{\overline{\omega}}(\overline{\infty}_-)$; see \cite[Lemma 5.1]{Poonen-McCal}. On the other hand, 
\[\overline{\omega}=\frac{(x^2+x)dx}{2y}=(2x+2x^3+2x^5+\dots )dx, \] 
and since the coefficient of $x^2$ in $\overline{\omega}$ is $0$, we see that the residue classes of $(0,2)$ and $(0,-2)$ each contain $\ord_{\overline{\omega}}(0,2)+1=2=\ord_{\overline{\omega}}(0,-2)+1$ rational points; see \cite[Remark 5.2]{Poonen-McCal}. This completes the proof of Lemma \ref{lem:d=1}.                                     
\end{proof} 
\begin{lem}{\label{lem:d=-1}} Let $C_{2}$ be the hyperelliptic curve given by \vspace{.05cm}  
\[C_{2}\,:\;\, y^2=-(x^8 + 12x^7 + 62x^6 + 182x^5 + 335x^4 + 398x^3 + 299x^2 + 131x + 25).\vspace{.05cm}  \] 
Then $C_{2}(\mathbb{Q})=\{(-1,\pm{1}), (-2,\pm{1})\}$ is a complete list of rational points.
\end{lem}
\begin{proof} In order to more easily compute with the Jacobian of $C_{2}$, we move two of the known rational points to infinity. Specifically, let 
\[ C_{2}':\;\, y^2=x^8 - x^7 - x^6 - 2x^5 - 5x^4 - 6x^3 - 6x^2 - 4x - 1.\]
Then $(x,y)\rightarrow(1/(x+1),y/(x+1)^4)$ is a birational map from $C_{2}$ to $C_{2}'$, and it suffices to show that 
\[C_{2}'(\mathbb{Q})=\{\infty_{+},\infty_-,(-1,\pm{1})\}\] 
to prove Lemma \ref{lem:d=-1}.

Let $J_{2}'$ be the Jacobian of $C_{2}'$. A two-descent with \texttt{Magma} shows that $J_{2}'(\mathbb{Q})$ has rank at most two. Moreover, we compute that 
\[\gcd\big(\#J_{2}'(\mathbb{F}_{11}),\#J_{2}'(\mathbb{F}_{29})\big)=1,\] and hence $J_{2}'(\mathbb{Q})$ has trivial torsion \cite[Appendix]{Katz}. On the other hand, the images of the rational points 
\[P_2':=\big[\infty_--\infty_+\big]\;\;\;\;\text{and}\;\;\;\; Q_2':=\big[(-1,1)-\infty_+\big]\] 
generate a subgroup isomorphic to $\mathbb{Z}/2\mathbb{Z}\times\mathbb{Z}/1396\mathbb{Z}$ in $J_{2}'( \mathbb{F}_{13})$. Therefore, $J_{2}'(\mathbb{Q})$ is free of rank two, and $P_2'$ and $Q_2'$ must generate a finite index subgroup. 

We proceed now with the method of Chabauty and Coleman at the prime $p=3$. In particular, we find three rational divisors  
\[R_1':=37P_2'+Q_2',\;\;R_2':=33Q_2'-3P_2',\;\;R_3':=31P_2'-35Q_2'\] 
in the kernel of the reduction homomorphism $J_{2}'(\mathbb{Q})\rightarrow J_{2}'(\mathbb{F}_3)$. Moreover, the Mumford representations $[a_i'(x),b_i'(x)]$ of $R_i'$ have the property that the cubic polynomials $a_i'(x)$ split completely in $\mathbb{Q}_3[x]$. Therefore, we have that 
\[R_i'=[p_{i1}'+p_{i2}'+p_{i3}'+\infty_-]-2[\infty_-+\infty_+]\]  
for some $p_{ij}'\in C_{2}'(\mathbb{Q}_3)$. In particular, the differences $D_1':=R_1'-R_2'$ and $D_2':=R_1'-R_3'$ are supported away from infinity, and 
\begin{align*}
D_1'=&\;[p_{11}'-p_{21}']+[p_{12}'-p_{22}']+[p_{13}'-p_{23}'],\\ 
D_2'=&\;[p_{11}'-p_{31}']+[p_{12}'-p_{32}']+[p_{13}'-p_{33}']   
\end{align*} 
are the sums of differences of $3$-adic points in the same residue class. We now use $D_1'$ and $D_2'$ to compute the annihilating $1$-form $\omega$ up to any precision. To do this, we again use the standard basis $\eta_k=x^k\,dx/2y$ for $0\leq k\leq 2$ of $\Omega_{C_{2}'}^1(\mathbb{Q}_3)$ and compute that \vspace{.075cm}
\[\Big(\int_0^{D_1'}\eta_k\Big)_{0\leq k\leq2}=\Big(\, \sum_{i=1}^3\int_{p_{2i}'}^{p_{1i}'}\eta_k \,\Big)_{0\leq k\leq2}=\big(\,3 + O(3^3),\;2\cdot3+2\cdot3^2 +O(3^4) ,\;2\cdot3 + 3^2 + O(3^4) \,\big)\] 
and that 
\[\Big(\int_0^{D_2'}\eta_k\Big)_{0\leq k\leq2}=\Big(\, \sum_{i=1}^3\int_{p_{3i}'}^{p_{1i}'}\eta_k \,\Big)_{0\leq k\leq2}=\big(\,3^2+O(3^3),\; O(3^3) ,\;2\cdot3^2+O(3^3)\,\big)\vspace{.1cm}\] 
with the \texttt{coleman_integral()} function in \texttt{Sage}. Writing $\omega=(c_2x^2+c_1x+c_0)\, dx/2y$ for some $c_0,c_1, c_2\in\mathbb{Z}_3$, we see that 
\[c_0+2c_1+2c_2\equiv0\Mod{3}\;\;\;\text{and}\;\;\; c_0+2c_2\equiv0\Mod{3}.\] 
Hence, $c_1\equiv0\Mod{3}$ and $c_0\equiv c_2\Mod{3}$, and up to an irrelevant scalar, $\omega$ reduces to 
\[\overline{\omega}=\frac{(x^2+1)dx}{2y}\in\Omega_{C_{2}'}^1(\mathbb{F}_3).\] 
On the other hand, \vspace{.075cm} 
\[C_{2}'(\mathbb{F}_3)=\{\overline{\infty}_{+},\overline{\infty}_-,(-1,\pm{1})\}. \vspace{.075cm}\] 
Therefore, \cite[Lemma 5.1]{Poonen-McCal} implies that $C_{2}'(\mathbb{Q})=\{\infty_{+},\infty_-,(-1,\pm{1})\}$ is a complete list of rational points, since $C_{2}'(\mathbb{Q})$ surjects onto $C_{2}'(\mathbb{F}_3)$ and $\overline{\omega}$ restricted to $C_2'(\mathbb{F}_3)$ is non-vanishing. 
\end{proof} 

\begin{lem}{\label{lem:d=5}} Let $C_{3}$ be the hyperelliptic curve given by \vspace{.05cm}  
\[C_{3}: y^2=5\cdot(x^8 + 12x^7 + 62x^6 + 182x^5 + 335x^4 + 398x^3 + 299x^2 + 131x + 25).\vspace{.05cm}  \] 
Then $C_{3}(\mathbb{Q})=\{(1,\pm{85})\}$ is a complete list of rational points.
\end{lem} 
\begin{proof} As in the previous case, it is useful to choose a model for $C_3$ with rational points at infinity. Specifically, let \vspace{.1cm}
\[C_3':\;\, y^2=7225x^8 + 23185x^7 + 32665x^6 + 26370x^5 + 13325x^4 + 4310x^3 + 870x^2 +100x + 5. \vspace{.1cm}\]
Then $(x,y)\rightarrow(1/(x-1),y/(x-1)^4)$ is a birational map from $C_{3}$ to $C_{3}'$, and it suffices to show that 
\[C_{3}'(\mathbb{Q})=\{\infty_{+},\infty_-\}\vspace{.1cm}\] 
to prove Lemma \ref{lem:d=5}. Let $J'_3$ be the Jacobian of $C_3'$. A two-descent with \texttt{Magma} shows that $J'_{3}(\mathbb{Q})$ has rank at most two. Moreover, we compute that \[\gcd\big(\#J'_{3}(\mathbb{F}_{3}),\#J'_{3}(\mathbb{F}_{11})\big)=1,\] 
and hence $J'_{3}(\mathbb{Q})$ has trivial torsion; see, for instance, \cite[Appendix]{Katz}. 

However, unlike the previous cases, the subgroup of $J'_3(\mathbb{Q})$ generated by the images of the known points from $C_3'(\mathbb{Q})$ do not generate a finite index subgroup. Therefore, we must search with a computer to find an additional rational divisor. With some difficulty, we accomplish this and find a new rational point with Mumford representation \cite{Mumford}: \vspace{.15cm} 
\[P_3':=\bigg[x^3 + \frac{849964}{706377}x^2 + \frac{43612}{100911}x + \frac{34501}{706377},\;-85x^4 +\frac{172717280}{2119131}x^2 + \frac{11001560}{302733}x + \frac{9205970}{2119131},\; 4\bigg].  \vspace{.15cm} \]
Moreover, we compute that the images of $P_3'$ and  
\[Q_3':=[\infty_--\infty_+]\]
in $J'_3(\mathbb{F}_{11})\times J'_3(\mathbb{F}_{13})$ generate a group isomorphic to $\mathbb{Z}/3\mathbb{Z}\times \mathbb{Z}/135450\mathbb{Z}$. Hence, $J'_3(\mathbb{Q})$ is free of rank two, and $P_3'$ and $Q_3'$ generate a finite index subgroup. Equivalently, the points \vspace{.05cm}
\[P_3=p_1+p_2+p_3+(1,85)-2\,[I_-+I_+]\;\;\;\text{and}\;\;\;\; Q_3=(1,85)-(1,-85) \vspace{.05cm}\]  
are independent points on the Jacobian $J_3$ of $C_3$; here $p_1$, $p_2$, and $p_3\in C_3(\overline{\mathbb{Q}})$ are points with $x$-coordinates supported on $p(x)=x^3 + x^2 - x - 4$, and $I_-$ and $I_+$ are the two points at infinity on $C_3$, defined over a quadratic extension. 

We now proceed with the method of Chabauty and Coleman at the prime $p=3$. Let \vspace{.05cm}
\[R_1':=-34Q_3',\;\;R_2':=15P_3'-7Q_3',\;\;R_3':=3R_2' \vspace{.05cm}\] 
be three rational divisors in the kernel of the reduction homomorphism $J'_{3}(\mathbb{Q})\rightarrow J'_{3}(\mathbb{F}_3)$. The Mumford representations $[a'_i(x),b'_i(x)]$ of $R'_i$ have the property that the cubic polynomials $a'_i(x)$ split completely in $\mathbb{Q}_3[x]$. Therefore, we have that \vspace{.05cm} 
\[R_i'=[p'_{i1}+p'_{i2}+p'_{i3}+\infty_-]-2[\infty_++\infty_-] \vspace{.05cm}\]  
for some $p'_{ij}\in C_{3}'(\mathbb{Q}_3)$. In particular, $D_1':=R'_1-R'_2$ and $D_2:=R'_1-R'_3$ are supported away from infinity, and 
\begin{align*}
D'_1=&\;[p'_{11}-p'_{21}]+[p'_{12}-p'_{22}]+[p'_{13}-p'_{23}],\\ 
D'_2=&\;[p'_{11}-p'_{31}]+[p'_{12}-p'_{32}]+[p'_{13}-p'_{33}]   
\end{align*} 
are the sums of differences of points in the same residue class. In particular, we use $D'_1$ and $D'_2$ to compute the annihilating $1$-form $\omega$ up to any precision. Again letting $\eta_k=x^k\,dx/2y$ for $0\leq k\leq 2$ of $\Omega_{C_{3}'}^1(\mathbb{Q}_3)$ be the standard basis, we compute with the \texttt{coleman_integral()} function in \texttt{Sage} that \vspace{.075cm}
\[\Big(\int_0^{D'_1}\eta_k\Big)_{0\leq k\leq2}=\Big(\, \sum_{i=1}^3\int_{p'_{2i}}^{p'_{1i}}\eta_k \,\Big)_{0\leq k\leq2}=\big(\,O(3^5),\; 3^2+2\cdot3^3 +O(3^4) ,\;2\cdot3^2 + O(3^4) \,\big).\] 
and that \vspace{.035cm} 
\[\Big(\int_0^{D'_2}\eta_k\Big)_{0\leq k\leq2}=\Big(\, \sum_{i=1}^3\int_{p'_{3i}}^{p'_{1i}}\eta_k \,\Big)_{0\leq k\leq2}=\big(\,2\cdot3 + 3^2 + O(3^4),\; 2\cdot3 +3^2+ 3^3+O(3^5),\;2\cdot3^2 + 3^4 +O(3^5)\,\big).\vspace{.1cm}\] 
Writing $\omega=(c_2x^2+c_1x+c_0)\, dx/2y$ for some coefficients $c_0,c_1, c_2\in\mathbb{Z}_3$, we see that \vspace{.1cm}  
\[c_1+2c_2\equiv0\Mod{3}\;\;\;\text{and}\;\;\; 2c_0+2c_1\equiv0\Mod{3}.\vspace{.1cm} \] 
Hence, $c_1\equiv c_2\Mod{3}$ and $c_0\equiv 2c_1\Mod{3}$, and up to an irrelevant scalar, $\omega$ reduces to 
\[\overline{\omega}=\frac{(x^2+x+2)dx}{2y}\in\Omega_{C_3'}^1(\mathbb{F}_3).\] 
Therefore, since $\overline{\omega}$ is non-vanishing on $C_3'(\mathbb{F}_3)$, it must be the case that each residue class of $C_3'(\mathbb{Q}_3)$ contains at most one rational point; see \cite[Lemma 5.1]{Poonen-McCal}. In particular, it suffices to show that the residue classes of $C_3'(\mathbb{Q}_3)$ over $(1,\pm{1})\in C_3'(\mathbb{F}_3)$, or equivalently the residue classes of $C_3(\mathbb{Q}_3)$ over $(2,\pm{1})\in C_3(\mathbb{F}_3)$, contain no rational points. To do this, we use the Mordell-Weil sieve \cite{MW-Sieve}. 

In its simplest form, the Mordell-Weil sieve is a procedure for ruling out rational points in residue classes in the following way: let $S$ be a set of primes of good reduction, let $N$ be an auxiliary integer, and consider the commutative diagram
\begin{displaymath}
    \xymatrixcolsep{5pc}\xymatrix{ \ar[d]^{\pi_{S}} C'_3(\QQ)\ar[r]^{\iota\;\;\;\;\,}&J_3'(\QQ)/NJ_3'(\QQ) \ar[d]^{\alpha_{S}} \\
               \displaystyle\prod_{q\in S}C_3'(\mathbb{F}_q)\ar[r]^{\beta_{S}\;\;\;\;\,}&\displaystyle\prod_{q\in S}J_3'(\mathbb{F}_q)/NJ_3'(\mathbb{F}_q)}
\end{displaymath}            
with the horizontal maps induced by the Abel-Jacobi map $w\rightarrow [w-\infty_+]$  and the vertical maps induced by reduction. Assuming we have generators of $J_3'(\QQ)$, we can compute the images of $\alpha_{S}$ and $\beta_{S}$ explicitly. Therefore, to rule out the existence of $p'\in C_3'(\QQ)$ such that $\pi_{q_0}(p')=\overline{p'}$ for some fixed $q_0\in S$ and some $\overline{p'}\in C_3'(\mathbb{F}_q)$, we just need to check that
\[\beta_{S}\Big(\{\overline{p'}_{q_0}\}\;\,\times \displaystyle\prod_{q\in S\,\mysetminus \{q_0\}}C_3'(\mathbb{F}_q)\Big)\;\bigcap\;\alpha_{S}\Big(J_3'(\QQ)\Big)=\varnothing.\]    
On the other hand, this remains true when instead of $J_3'(\QQ)$ we work with a subgroup $G_3'$ of finite index such that the index is coprime to $\#J_3'(\mathbb{F}_q)$ for $q\in S$. This property can be checked for any given potential prime divisor $\ell$ of the index by verifying that
\[G'/\ell G'\rightarrow \prod_{\ell'\in S_\ell} J_3'(\mathbb{F}_{\ell'})/\ell J_3'(\mathbb{F}_{\ell'})\]
is injective for some set $S_\ell$ of good primes. In particular, for $C_3'$ the set of primes 
\[S:=\{3, 11, 13, 17, 19, 61, 331, 379, 409, 947, 1201\}\]
and the integer
\[N=28745640=2^3\cdot3^2\cdot5\cdot7\cdot11\cdot17\cdot61\vspace{.075cm}\]
are sufficient to prove that the only residue classes in $C_3'(\mathbb{F}_3)$ that contain rational points are the points at infinity; here the set $S$ was chosen since the prime divisors $\ell$ of $J_3'(\mathbb{F}_p)$ are small, i.e $\ell\leq67$ for all $p\in S$. The computational details for this argument can be found in our file \texttt{Hindes-Curve.Magma} at the link above. This completes the proof of Lemma \ref{lem:d=5}.       
\end{proof} 
\begin{lem}{\label{lem:d=-5}} Let $C_{4}$ be the hyperelliptic curve given by \vspace{.075cm}  
\[C_{4}: y^2=-5\cdot(x^8 + 12x^7 + 62x^6 + 182x^5 + 335x^4 + 398x^3 + 299x^2 + 131x + 25).\vspace{.075cm}  \] 
Then $C_{4}(\mathbb{Q})=\{(-7/3,\pm{185}/81)\}$ is a complete list of rational points.
\end{lem}  
\begin{proof}
Let $J_{4}$ be the Jacobian of $C_{4}$. As in the previous case, it is useful to pass back and forth between $C_{4}$ and a different model with rational points at infinity. Specifically, let \vspace{.15cm} 
\[C_{4}': \;\,y^2=\frac{34225}{6561}x^8 - \frac{22895}{2187}x^7 + \frac{29515}{729}x^6 - \frac{4910}{243}x^5 - \frac{5825}{81}x^4 + \frac{3430}{27}x^3 - \frac{830}{9}x^2 + \frac{100}{3}x - 5 .\vspace{.15cm}\]
Then $(x,y)\rightarrow(1/(x+7/3),y/(x+7/3)^4)$ is a birational map from $C_{4}$ and $C_{4}'$, and it suffices to show that 
\[C_{4}'(\mathbb{Q})=\{\infty_{+},\infty_-\} \vspace{.1cm}\] 
to prove Lemma \ref{lem:d=-5}. Let $J'_{4}$ be the Jacobian of $C_{4}'$. A two-descent with \texttt{Magma} shows that $J'_{4}(\mathbb{Q})$ has rank at most two. Moreover, we compute that 
\[\gcd\big(\#J'_{4}(\mathbb{F}_{11}),\#J'_{4}(\mathbb{F}_{43})\big)=1, \vspace{.075cm}\] 
and hence $J'_{4}(\mathbb{Q})$ has trivial torsion  \cite[Appendix]{Katz}. Now let \vspace{.05cm} 
\[P_4=p_1+p_2+p_3+\Big(\frac{-7}{3},\frac{185}{81}\Big)-2(I_-+I_+)\;\;\;\text{and}\;\;\;\; Q_4=\Big(\frac{-7}{3},\frac{185}{81}\Big)-\Big(\frac{-7}{3},\frac{-185}{81}\Big) \vspace{.1cm}\] 
be rational divisors in $J_{4}(\mathbb{Q})$; here $p_1$, $p_2$, and $p_3\in C_{4}(\overline{\mathbb{Q}})$ are points with $x$-coordinates supported on $p(x)=x^3 + 11/3x^2 + 38/7x + 23/7$, and $I_-$ and $I_+$ are the two points at infinity on $C_{4}$, defined over a quadratic extension (once again, the point $P_4$ was found using a computer search). 

Let $P_4'$ and $Q_4'$ denote the images of $P_4$ and $Q_4$ in $J_{4}'(\mathbb{Q})$ respectively (their Mumford representations are quite large). Then $P_4'$ and $Q_4'$ generate a group isomorphic to $\mathbb{Z}/3\mathbb{Z}\times \mathbb{Z}/531\mathbb{Z}$ in $J'_{4}(\mathbb{F}_{11})$. Hence, $J'_{4}(\mathbb{Q})$ is free of rank two, and $P_4'$ and $Q_4'$ generate a finite index subgroup.    

We now proceed with the method of Chabauty and Coleman at the prime $p=3$. Let \vspace{.075cm}
\[R_1':=2P_4'+9Q_4',\;\;R_2':=19P_4'+100Q_4',\;\;R_3':=-28P_4'-51Q_4'.\vspace{.075cm}\] 
The Mumford representations $[a'_i(x),b'_i(x)]$ of $R'_i$ have the property that the cubic polynomials $a'_i(x)$ split completely in $\mathbb{Q}_3[x]$. Therefore, we have that \vspace{.075cm} 
\[R_i'=[p'_{i1}+p'_{i2}+p'_{i3}+\infty_-]-2[\infty_++\infty_-] \vspace{.075cm}\]  
for some $p'_{ij}\in C_{4}'(\mathbb{Q}_3)$. In particular, one checks that the divisors $D_1':=R'_1-R'_2$ and $D_2:=R'_1-R'_3$ are supported away from infinity, in the kernel of the map, and\vspace{.075cm}  
\begin{align*}
D'_1=&\;[p'_{11}-p'_{21}]+[p'_{12}-p'_{22}]+[p'_{13}-p'_{23}],\\ 
D'_2=&\;[p'_{11}-p'_{31}]+[p'_{12}-p'_{32}]+[p'_{13}-p'_{33}]\vspace{.05cm}   
\end{align*} 
are the sums of differences of points in the same residue class. 

However, since $C_{4}'$ has bad reduction at $p=3$, it is easier for computational purposes to push these divisors back to $J_{4}$, which has good reduction. Let $D_1$ and $D_2\in J_{4}(\mathbb{Q})$ denote the images of $D_1'$ and $D_2'$ respectively. Letting $\eta_k=x^k\,dx/2y$ for $0\leq k\leq 2$ of $\Omega_{C_{4}}^1(\mathbb{Q}_3)$, we compute with the \texttt{coleman_integral()} function in \texttt{Sage} that \vspace{.075cm}
\[\Big(\int_0^{D_1}\eta_k\Big)_{0\leq k\leq2}=\Big(\, \sum_{i=1}^3\int_{p_{2i}}^{p_{1i}}\eta_k \,\Big)_{0\leq k\leq2}=\big(\,2\cdot3+3^2+O(3^3),\; 3^2 +O(3^3) ,\;3^2 + O(3^3) \,\big).\] 
and that 
\[\Big(\int_0^{D_2}\eta_k\Big)_{0\leq k\leq2}=\Big(\, \sum_{i=1}^3\int_{p_{3i}}^{p_{1i}}\eta_k \,\Big)_{0\leq k\leq2}=\big(\,O(3^3),\; 3^2 + O(3^5),\; 2\cdot3^2+O(3^3)\,\big).\vspace{.1cm}\] 
Writing $\omega=(c_2x^2+c_1x+c_0)\, dx/2y$ for some $c_0,c_1, c_2\in\mathbb{Z}_3$, we see that \vspace{.05cm} 
\[c_0\equiv0\Mod{3}\;\;\;\text{and}\;\;\; c_1+2c_2\equiv0\Mod{3}.\] 
Hence, $c_1\equiv c_2\Mod{3}$, and up to an irrelevant scalar, $\omega$ reduces to 
\[\overline{\omega}=\frac{(x^2+x)dx}{2y}\in\Omega_{C_{4}}^1(\mathbb{F}_3).\]  
Therefore, \cite[Lemma 5.1]{Poonen-McCal} implies that each residue class at infinity contains at most one rational point. In particular, it suffices to show that the residue classes of $C_{4}(\mathbb{Q}_3)$ over $(0,\pm{1})\in C_{4}(\mathbb{F}_3)$ contain no rational points. To do this, we use the Mordell-Weil sieve; moreover, we can provide more details in this case, since it is less complicated than the procedure for $C_3$. 

Let $G_4$ be the subgroup of $J_4(\QQ)$ generated by the divisors $P_4$ and $Q_4$ above. Since we cannot be sure that we capture the full Mordell-Weil group with $G_4$, we first show that the index $[J_4(\QQ):G_4]$ is not divisible by the small primes in $S'=\{2,3,5,11\}$. This is relatively easy: for each $\ell\in S'$, we produce an auxiliary set of primes $S_\ell$ such that the induced map 
\[G_4/\ell G_4\rightarrow\prod_{\ell'\in S_\ell} J_4(\mathbb{F}_{\ell '})/ \ell J_4(\mathbb{F}_{\ell'})\]
is injective. It is straightforward to verify with \texttt{Magma} that the sets $S_2=\{3,13\}$, $S_3=\{3,23\}$, $S_5=\{3,43\}$, and $S_{11}=\{3,7,17\}$ satisfy this property. In particular, if $\overline{(G_4)}_q$ and $\overline{J_4(\QQ)}_q$ denote the images of $G_4$ and $J_4(\QQ)$ in $J_4(\mathbb{F}_q)$ respectively, then it follows from our exclusion of the small indices in $S'$ that $\overline{(G_4)}_q=\overline{J_4(\QQ)}_q$ for all $q\in S=\{3,7\}$; the upshot of this step is that it allows us to be sure that any local information gained by reducing $J_4(\mathbb{Q})$ modulo $q\in S$ is captured instead by reducing $G_4$, which is concrete and explicitly known.

We now run the Mordell-Weil sieve at the primes in $S=\{3,7\}$ and the integer $N=0$. That is, we consider the commutative diagram \begin{displaymath}
    \xymatrix{ \ar[d]^{\pi_{S}} C_4(\QQ)\; \ar[r]^{\iota} & \;J_4(\QQ) \ar[d]^{\alpha_{S}} \\
                C_4(\mathbb{F}_3)\times C_4(\mathbb{F}_7)\,\ar[r]^{\beta_{S}} & \,J_4(\mathbb{F}_3)\times J_4(\mathbb{F}_7) }
\end{displaymath}            
with horizontal maps induced by $w\rightarrow[2w-(I_-+I_+)]$ and vertical maps induced by reduction. In particular, we rule out the existence of rational points $w\in C_4(\QQ)$ in the residue classes over $(0,\pm{1})\in C_4(\mathbb{F}_3)$ by verifying with \texttt{Magma} that 
\[\beta_{S}\Big(\{(0,\pm{1})\}\;\,\times C_4(\mathbb{F}_7)\Big)\;\bigcap\;\alpha_{S}\big(G_4\big)=\varnothing.\]
This completes the proof of Lemma~\ref{lem:d=-5}.     
\end{proof}  
\section{Dynamical Galois Groups: $\phi_c(x)=x^2+c$ \textup{and} $b=0$}{\label{sec:rationalpoints2}
We now use Corollary \ref{cor:curve} to prove Theorem \ref{thm:small5th}. Specifically, let $\phi(x)=\phi_t(x)=x^2+t$ and let $n=5$; here again, $t$ is an indeterminate. Then we have 
\begin{align*}
-\phi(0)&=-t,\\
\phi^2(0)&=t^2 + t,\\
\phi^3(0)&=(t^2+t)^2+t,\\
\phi^4(0)&= ((t^2+t)^2+t)^2+t,\\  
\phi^5(0)&= (((t^2+t)^2+t)^2+t)^2+t. \vspace{.1cm} 
\end{align*} 
In particular, we must compute the integral points on the (singular) curves, 
\[C_{\phi,5}^{(\epsilon)}: y^2=-\phi(0)^{\epsilon_1}\cdot\phi^2(0)^{\epsilon_2}\cdot \phi^3(0)^{\epsilon_3}\cdot \phi^{4}(0)^{\epsilon_{4}}\cdot\phi^5(0)\]
for $\epsilon\in\mathbb{F}_2^{4}$. However, unlike the corresponding polynomials in section \ref{sec:rationalpoints}, these polynomials have common factors, which we must divide out appropriately, before we can compute integral points. Moreover, there is also the problem that, apriori, we must compute the integral points on $16$ curves. However, because the orbit $\Orb_{\phi}(0)=\{\phi^n(0)\}_{n\geq1}$ is a rigid divisibility sequence \cite[Lemma 12]{Rafe-Spencer}, we can cut down the number of curves to $2$: \vspace{.03cm}     
\begin{thm}{\label{thm:prime}} Let $\phi_c(x)=x^2+c$ for some $c\in\mathbb{Z}$ and let $p\geq3$ be an odd prime. If $c$ does not give an integral point on either of the hyperelliptic curves \vspace{.1cm}
\[\mathcal{C}_{p}:\, y^2=1/t\cdot \phi_t^p(0)\;\;\;\;\;\text{and}\;\;\;\;\;\mathcal{C}_p^{(-1)}:\, -y^2=1/t\cdot \phi_t^p(0),\vspace{.125cm}\] 
then $G_{p-1}(\phi_c)=\Aut(T_{p-1})$ implies $G_{p}(\phi_c)=\Aut(T_{p})$. \vspace{.015cm}  
\end{thm} 
\begin{proof} It follows from \cite[Theorem 1.7]{Stoll-Galois} and its proof that if $G_{n-1}(\phi_c)=\Aut(T_{2,n-1})$ and $G_{n}(\phi_c)\neq\Aut(T_{2,n})$, then $\pm{b}_n$ is a square (in the notation of \cite{Stoll-Galois}). Moreover, if $n=p$ is prime, then $b_p=-\phi_c^p(0)/c$ and the claim follows.  
\end{proof} 
We are now ready to finish the proof of Theorem \ref{thm:small5th}. To do this, we succeed in finding all rational points on a hyperelliptic curve of genus $7$, whose Jacobian has rank $5$, a task which might be some independent interest for those studying computational arithmetic geometry.     
\begin{proof}[(Proof of Theorem \ref{thm:small5th})] It follows from \cite[Theorem 1.1]{Me:4th}, Corollary \ref{cor:curve} and Theorem \ref{thm:prime}, that to prove Theorem \ref{thm:small5th}, we must classify the integral points on the hyperelliptic curves 
\[ \mathcal{C}_1:\; y^2=g(t)=1/t\cdot \phi_t^5(0), \;\;\;\;\; \mathcal{C}_2:\; y^2=-g(t)=-1/t\cdot \phi_t^5(0)\]
of genus $7$. In fact, we prove that there are no unknown rational points on these curves and catalogue this information in Lemma \ref{lem:5th1} and Lemma \ref{lem:5th2} below. However, in the interest of keeping these results self-contained, we use standard notation for hyperelliptic curves and replace the variable $t$ with the (more standard) variable $x$ in what follows.  
\end{proof} 
\begin{lem}{\label{lem:5th1}} Let $\mathcal{C}_1$ be the hyperelliptic curve given by \vspace{.15cm}
\begin{align*} \mathcal{C}_1:\; y^2=x^{15} + 8x^{14} + &28x^{13} + 60x^{12} + 94x^{11} + 116x^{10} + 114x^9 + 94x^8+\\ 
&69x^7 + 44x^6 + 26x^5 + 14x^4 + 5x^3 + 2x^2 + x + 1.
\end{align*} 
Then $\mathcal{C}_1(\mathbb{Q})=\{\infty, (-1,\pm{1}), (0,\pm{1})\}$ is a complete set of rational points. 
\end{lem} 
\begin{proof} Let $\mathcal{J}_1$ be the Jacobian $\mathcal{C}_1$. A two-descent with \texttt{Magma} shows that $\mathcal{J}_1(\mathbb{Q})$ has rank at most five. Moreover, we compute that 
\[\gcd\big(\#\mathcal{J}_{1}(\mathbb{F}_{7}),\#\mathcal{J}_{1}(\mathbb{F}_{19})\big)=1,\]
and hence $\mathcal{J}_1(\QQ)$ has trivial torsion \cite[Appendix]{Katz}. On the other hand, consider the five rational divisors with Mumford representations: 
\begin{align*} 
&\mathcal{P}_1:=[x+1,-1],\;\;\;\, \mathcal{Q}_1:=[x,1], \;\;\;\, \mathcal{R}_1:=[x^2 + 1, -x],
\\
&\\ 
&\mathcal{U}_1:=[x^2 + x - 1, -2x - 3],\;\; \mathcal{V}_1:=[x^4 - x - 1,10x^3 + 14x^2 + 16x + 9], 
\end{align*}  
and let $\mathcal{G}_1$ be the subgroup they generate in $\mathcal{J}_1(\mathbb{Q})$. Then we compute that the image of $\mathcal{G}_1$ via the reduction homomorphism 
\[\mathcal{J}_1(\mathbb{Q})\;\longrightarrow\; \prod_{\mathclap{\;\;\;\;\;\;\;\;\;\;\;\;\;\; p\in \{11,17,19,41,43,53,59,61\}}}\;\mathcal{J}_1(\mathbb{F}_{p})\] 
is isomorphic to 
\begin{align*}
&\hspace{.5cm}\Scale[0.95]{\mathbb{Z}/4\mathbb{Z}\times \mathbb{Z}/4\mathbb{Z} \times \mathbb{Z}/12\mathbb{Z} \times \mathbb{Z}/30192\mathbb{Z}\; \times}\\ 
&\hspace{.5cm}\Scale[0.95]{\mathbb{Z}\big/266359966347730293793903674442891487539799606677968895688846898339666912\mathbb{Z}},
\end{align*} 
which needs five generators. Therefore, since $\mathcal{J}_1(\mathbb{Q})$ has trivial torsion, $\mathcal{J}_1(\mathbb{Q})$ must have rank $5$, and $\mathcal{G}_1$ is a finite index subgroup. 

We proceed now with the method of Chabauty and Coleman at the prime $p=3$. To find divisors suitable for calculation, we look within the kernel of the map $\mathcal{J}_1(\mathbb{Q})\rightarrow\mathcal{J}_1(\mathbb{F}_3)$ given by reduction. In particular, we use the LLL algorithm to compute a short basis of this kernel:
\begin{align*} 
\mathcal{D}_{1}:=2&\mathcal{Q}_1+3\mathcal{U}_1+2\mathcal{V}_1,\;\;\;\, \mathcal{D}_2:=4\mathcal{P}_1+\mathcal{Q}_1-2\mathcal{R}_1+\mathcal{U}_1-3\mathcal{V}_1,\;\;\;\,\mathcal{D}_3:=4\mathcal{P}_1+\mathcal{Q}_1+4\mathcal{R}_1-\mathcal{U}_1,\\
&\\
& \mathcal{D}_4:=3\mathcal{P}_1+\mathcal{Q}_1-5\mathcal{R}_1-\mathcal{U}_1+3\mathcal{V}_1,\,\;\;\mathcal{D}_5:=\mathcal{P}_1-6\mathcal{Q}_1+\mathcal{R}_1+2\mathcal{U}_1+3\mathcal{V}_1.\vspace{.1cm}
\end{align*} 
Let $\pi_3:\mathcal{C}_1(\overline{\mathbb{Q}})\rightarrow\mathcal{C}_1(\overline{\mathbb{F}}_3)$. Then we compute a decomposition of the divisors $\mathcal{D}_i$ satisfying:    
\[\;\;\mathcal{D}_i=\sum_{j=1}^7x_{ij} - 7\,(-1,1), \;\;\;\;\;\;\;\text{for some}\;\, x_{ij}\in\mathcal{C}_1(\overline{\mathbb{Q}}),\;\;\; \pi_3(x_{ij})=(-1,1).\]
In particular, we can compute the integrals of the $\mathcal{D}_i$ on the standard basis $\eta_k=x^k\,dx/2y$ for $0\leq k\leq 6$ of $\Omega_{\mathcal{C}_{1}}^1(\mathbb{Q}_3)$ as a sum of ``tiny integrals"  \cite[\S3.2]{Jennifer} using a uniformizer in the residue class of $(-1,1)$. Explicitly,  
\[\Big(\int_{0}^{\mathcal{D}_j}\eta_i\Big)_{0\leq i\leq 6, 1\leq j\leq 5}=\Scale[0.73]{
\begin{pmatrix}
-62\cdot3 + O(3^6) &\, 3^2 + O(3^6) &\, -20\cdot3^2 + O(3^6) & \,O(3^6) &\, 32\cdot3^2 + O(3^6)\\ 
-106\cdot3 + O(3^6) &\, 34\cdot3^2 + O(3^6) &\, 109\cdot3 + O(3^6) &\, 73\cdot3 + O(3^6) &\, -119\cdot3 + O(3^6)\\
-74\cdot3 + O(3^6) &\, 116\cdot3 + O(3^6) &\, 50\cdot3 + O(3^6) &\, -20\cdot3 + O(3^6) &\, -79\cdot3 + O(3^6)\\
-35\cdot3^2 + O(3^6) &\, 4\cdot3^3 + O(3^6) &\, 106\cdot3 + O(3^6) &\, 31\cdot3 + O(3^6) &\, 50\cdot3 + O(3^6)\\
-101\cdot3 + O(3^6)&\, -47\cdot3 + O(3^6) &\, 85\cdot3 + O(3^6) &\, -11\cdot3 + O(3^6) &\, -4\cdot3^4 + O(3^6)\\
37\cdot3 + O(3^6) &\, -101\cdot3 + O(3^6) &\, -11\cdot3 + O(3^6) &\, -89\cdot3 + O(3^6) &\, -61\cdot3 + O(3^6)\\
101\cdot3 + O(3^6) &\, 41\cdot3 + O(3^6) &\, 31\cdot3^2 + O(3^6) &\, -16\cdot3 + O(3^6) &\, 2\cdot3 + O(3^6)
\end{pmatrix}} \vspace{.1cm}\]
This information implies that the reductions mod $3$ of the (suitably scaled) annihilating differentials of $\mathcal{J}_1(\QQ)$ fill the subspace of $\Omega_{\mathcal{C}_{1}}^1(\mathbb{F}_3)$ spanned by 
\[\overline{\omega}_1=\big(x^6 + x^3 + 2x^2 + x\big)\frac{dx}{2y} \;\;\;\;\;\; \text{and}\;\;\;\;\;\;\; \overline{\omega}_2=(x^5 + 2x^4 + x^3 + 2x + 2)\frac{dx}{2y}.\;\;\;\;\;\;\] 
On the other hand, the reduction map $\pi_3: \mathcal{C}_1(\QQ)\rightarrow\mathcal{C}_1(\mathbb{F}_3)$ on rational points is a surjection:
\[\mathcal{C}_1(\mathbb{F}_3)=\{\overline{\infty}, (0,\pm{1}),(-1,\pm{1})\}.\]
Moreover, each residue class contains exactly one rational point, since $\overline{\omega}_1$ is non-vanishing on $\overline{\infty}$ and $(-1,\pm{1})$ and $\overline{\omega}_2$ is non-vanishing on $(0,\pm{1})$; see \cite[Lemma 5.1]{Poonen-McCal}. In particular, we have indeed found a complete list of rational points for $\mathcal{C}_1$.      
\end{proof}
 \begin{lem}{\label{lem:5th2}} Let $\mathcal{C}_2$ be the hyperelliptic curve given by \vspace{.15cm}
\begin{align*} 
\mathcal{C}_2:\; y^2=-(x^{15} + 8x^{14} + 28x^{13} + &60x^{12} + 94x^{11} + 116x^{10} + 114x^9 +94x^8+ \\  
&69x^7 + 44x^6 + 26x^5 + 14x^4 + 5x^3 + 2x^2 + x + 1).
\end{align*} 
Then $\mathcal{C}_2(\mathbb{Q})=\{\infty, (-2,\pm{1})\}$ is a complete set of rational points. 
\end{lem}
\begin{proof} Let $\mathcal{J}_2$ be the Jacobian $\mathcal{C}_2$. A two-descent with \texttt{Magma} shows that $\mathcal{J}_2(\mathbb{Q})$ has rank at most three. Moreover, we compute that 
\[\gcd\big(\#\mathcal{J}_{1}(\mathbb{F}_{5}),\#\mathcal{J}_{1}(\mathbb{F}_{7})\big)=1,\]
and hence $\mathcal{J}_2(\QQ)$ has trivial torsion  \cite[Appendix]{Katz}. On the other hand, consider the three rational divisors with Mumford representations: \vspace{.1cm} 
\[\mathcal{P}_2:=[x+2,1],\;\;\;\, \mathcal{Q}_2:=[x^2+1,1], \;\;\;\, \mathcal{R}_2:=[x^3 + x^2 - x + 1,x^4 + 2x^3 + x^2 + x],  \vspace{.1cm} \]  
and let $\mathcal{G}_2$ be the subgroup they generate in $\mathcal{J}_2(\mathbb{Q})$. Then we compute that the image of $\mathcal{G}_2$ via the reduction homomorphism 
\[\mathcal{J}_2(\mathbb{Q})\rightarrow \mathcal{J}_2(\mathbb{F}_{3})\times\mathcal{J}_2(\mathbb{F}_{5})\times\mathcal{J}_2(\mathbb{F}_{7})\times\mathcal{J}_2(\mathbb{F}_{11})\times \mathcal{J}_2(\mathbb{F}_{19}),\] 
is isomorphic to
\[ \mathbb{Z}/3\mathbb{Z}\times \mathbb{Z}/18\mathbb{Z} \times \mathbb{Z}/1461606886871374660285051800\mathbb{Z},\]
which needs three generators. Therefore, since $\mathcal{J}_2(\mathbb{Q})$ has trivial torsion, $\mathcal{J}_2(\mathbb{Q})$ must have rank three, and $\mathcal{G}_2$ is a finite index subgroup.

We proceed now with the method of Chabauty and Coleman at the prime $p=3$ and look within the kernel of the reduction homomorphism $\mathcal{J}_2(\mathbb{Q})\rightarrow\mathcal{J}_2(\mathbb{F}_3)$ for suitable divisors. In particular, we use the LLL algorithm to compute a short basis of the kernel:
\[ \mathcal{D}_1=3\mathcal{P}_2-4\mathcal{Q}_2-8\mathcal{R}_2,\;\;\;\; \mathcal{D}_2=8\mathcal{P}_2-8\mathcal{Q}_2+7\mathcal{R}_2,\;\;\;\; \mathcal{D}_3=13\mathcal{P}_2+16\mathcal{Q}_2.\] 
Let $\pi_3:\mathcal{C}_2(\overline{\mathbb{Q}})\rightarrow\mathcal{C}_2(\overline{\mathbb{F}}_3)$ be the reduction map. Then we compute a decomposition of the divisors $\mathcal{D}_i$ satisfying:    
\[\;\;\mathcal{D}_i=\sum_{j=1}^7x_{ij} - 7\,(-2,1), \;\;\;\;\;\;\;\text{for some}\;\, x_{ij}\in\mathcal{C}_2(\overline{\mathbb{Q}}),\;\;\; \pi_3(x_{ij})=(-2,1).\] 
In particular, we compute the integrals of the $\mathcal{D}_i$ on the standard basis $\eta_k=x^k\,dx/2y$ for $0\leq k\leq 6$ of $\Omega_{\mathcal{C}_{2}}^1(\mathbb{Q}_3)$ as a sum of ``tiny integrals" \cite[\S3.2]{Jennifer} using a uniformizer in the residue class of $(-2,1)$. Explicitly,  
\[\Big(\int_{0}^{\mathcal{D}_j}\eta_i\Big)_{0\leq i\leq 6, 1\leq j\leq 3}=
\begin{pmatrix}
95\cdot3 + O(3^6)&\, 113\cdot3 + O(3^6) &\, -32\cdot 3 + O(3^6)\\ 
-83\cdot3 + O(3^6)&\, 65\cdot3 + O(3^6) &\, -7\cdot3^2 + O(3^6)\\ 
38\cdot3^2 + O(3^6) &\, -97\cdot3 + O(3^6) &\, 44\cdot3 + O(3^6)\\
118\cdot3 + O(3^6) &\, -41\cdot3 + O(3^6) &\, -59\cdot3 + O(3^6)\\
-74\cdot3 + O(3^6)&\, 7\cdot3 + O(3^6)&\, 64\cdot3 + O(3^6)\\
97\cdot3 + O(3^6)&\, 65\cdot3 + O(3^6)&\, 23\cdot3 + O(3^6)\\
-31\cdot3^2 + O(3^6) &\, -44\cdot3 + O(3^6) &\, -20\cdot3 + O(3^6)\\
\end{pmatrix} \vspace{.1cm}\] 
Therefore, the mod $3$ reductions of the (suitably scaled) annihilating differentials of $\mathcal{J}_2(\QQ)$ fill the subspace of $\Omega_{\mathcal{C}_{2}}^1(\mathbb{F}_3)$ spanned by \vspace{.05cm} 
\[\overline{\omega}_1=\big(x^2 + x + 1\big)\frac{dx}{2y},\;\;\;\, \overline{\omega}_2=\big(x^4 - x^3\big)\frac{dx}{2y},\;\;\;\, \overline{\omega}_3=\big(x^5 - x^3 - x - 1\big)\frac{dx}{2y},\;\;\;\, \overline{\omega}_4=\big(x^6 - x - 1\big)\frac{dx}{2y}.\vspace{.1cm}\]  
Furthermore, as in the previous case, the reduction map $\pi_3: \mathcal{C}_2(\QQ)\rightarrow\mathcal{C}_2(\mathbb{F}_3)$ on rational points is a surjection:
\[\mathcal{C}_2(\mathbb{F}_3)=\{\overline{\infty},(-2,\pm{1})\}.\]
Moreover, each residue class contains exactly one rational point, since in particular, $\overline{\omega}_4$ is non-vanishing on all of $\mathcal{C}_2(\mathbb{F}_3)$; see \cite[Lemma 5.1]{Poonen-McCal}.   
\end{proof} 
\section{Other rigid properties of arboreal representations}{\label{sec:irre}}
As we discussed in the introduction, the Vojta conjecture for curves predicts a ``surjective rigidity" in arboreal representations within one-dimensional families; see Theorem \ref{rigid}. It is therefore natural to ask if one expects this sort of behavior to hold over all monic, quadratic polynomials with integer coefficients (a $2$-dimensional family). At present, such a statement is unclear in either direction. However, if one instead considers only irreducibility (not the full Galois behavior), then there are a strikingly small number of counterexamples to this sort of phenomenon, even at low stages of iteration. For instance, in a search for integer pairs $(c, b)$ with $|b|\leq10^6$, satisfying
\[ G_2(\phi_c,b)\;\text{is transitive and}\; G_3(\phi_c,b)\;\text{is not transitive},\] 
one finds only ten examples; note here that only the basepoint $b$ is bounded, i.e. $c$ can be any integer. Moreover, if one looks at the analogous statement for the third and fourth stage of iteration, one finds no examples.     
\begin{prop}{\label{prop:3rd}} Let $\phi_c(x)=x^2+c$ for $c\in\mathbb{Z}$, and let $b\in\mathbb{Z}$ be a basepoint with $|b|\leq10^6$. Then the following statements hold:
\begin{enumerate}[topsep=10pt, partopsep=10pt, itemsep=12pt]  
\item[\textup{(1)}] $\phi_c^2(x)-b$ is irreducible if and only if $\phi_c^3(x)-b$ is irreducible, unless $(b,c)$ is in: \vspace{.175cm} 
\begin{align*}
\;\;\,R:=\big\{(-5,-1),\,(-32&5,-1),\, (-2041,7),\,(10674,-18),\,(-16400,-16),\, (45040,-16),\,\\
\,&\;\;(-204784,16),\, (-204793,7),\, (-238158,-18),\, (-777925,-1)\big\}. 
\end{align*}   
\item[\textup{(2)}] $\phi_c^3(x)-b$ is irreducible if and only if $\phi_c^4(x)-b$ is irreducible. \vspace{.175cm} 
\end{enumerate}
Therefore, if $(b,c)\not\in R$ and $\phi_c^2(x)-b$ is irreducible, then $\phi_c^4(x)-b$ is irreducible.   
\end{prop}
\begin{proof} By a change of variables, the polynomial $\phi_c^n(x)-b$ is irreducible if and only if the polynomial $\varphi_{\gamma,c'}^n(x)$ is irreducible; here as usual, $\varphi_{\gamma,c'}(x)=(x-\gamma)^2+c'$, where $\gamma=-b$ and $c'=c-b$. Therefore, after making the substitutions $b\leftrightarrow\gamma$ and $c\leftrightarrow c'$, it suffices to classify pairs:      
\[S_3:=\big\{(\gamma,c)\in\mathbb{Z}\times\mathbb{Z}\,:\,|\gamma|\leq10^6,\;\, \varphi_{\gamma,c}^2(x)\,\text{$=$ irreducible},\; \text{and}\;\,\varphi_{\gamma,c}^3(x)\,\text{$=$ reducible}\big\}.\]
A priori, for some $\gamma$ there could be infinitely many $c$'s such $(\gamma,c)\in S_3$. However, if $(\gamma,c)\in S_3$, then $\varphi_{\gamma,c}^3(\gamma)$ is a square in $\mathbb{Z}$; see Proposition \ref{lem:irre} above. Therefore, if we take the rational points approach (as we've done elsewhere), then we see that $(c,y)$ gives an integral point on the elliptic curve $E_{\gamma}: y^2=\varphi_{\gamma,c}^3(\gamma)$ for some $y\in\mathbb{Z}$. Hence, the set of such $c$'s is finite by a theorem of Siegel. However, the best known bounds for integral points on elliptic curves are doubly exponential in the height of the coefficients of the defining equation \cite[\S IX.5]{Silverman:ellip}, and therefore a search for $c\in S_3$ with this approach becomes impractical for large $\gamma$. On the other hand, if $\varphi_{\gamma,c}^3(\gamma)$ is a square in $\mathbb{Z}$, then \cite[Lemma 4.3]{Rafe:LMS} implies that  
\[ |c-\gamma|\leq1+\sqrt{|\gamma|+1},\] 
a much smaller bound for $c$ in terms of $\gamma$. After two weeks of searching with \texttt{Magma} along these lines, we exhaust all possibilities for $(\gamma,c)$ and find the complete list in Proposition \ref{prop:3rd} statement (1). Moreover, an analogous search yields statement (2).   
\end{proof} 
Since irreducibility of the third iterate implies irreducibility of the fourth iterate in all known examples, we conjecture that this is always the case: 
\begin{conj} Let $\phi_c(x)=x^2+c$ for $c\in\mathbb{Z}$, and let $b\in\mathbb{Z}$ be any basepoint. Then $\phi_c^3(x)-b$ is irreducible if and only if $\phi_c^4(x)-b$ is irreducible. 
\end{conj} 
This problem is related to (but not entirely explained by) the fact that the surfaces, 
\[\mathcal{S}_n: \big\{(\gamma,y,c): y^2=\varphi_{\gamma,c}^n(\gamma)\big\},\]
have very few integral points; for instance, most of the integral points on the elliptic surface $\mathcal{S}_3$ do not come from quadratic polynomials whose third iterate is the first to be reducible.    

\end{document}